\documentclass{amsproc}

\usepackage[margin=1in]{geometry}
\usepackage{amsmath,amsthm,amssymb,amsrefs,bbm,color,enumitem,esint,esvect,float,graphicx,mathrsfs}
\usepackage{euscript,latexsym}
\usepackage{accents}
\usepackage{epstopdf}
\usepackage{graphicx}
\graphicspath{ {./Documents/} }
\AppendGraphicsExtensions{.tif}

\newtheorem{theorem}{Theorem}[section]
\newtheorem{lemma}[theorem]{Lemma}
\newcommand{\ii}{\mathrm{i}}

\theoremstyle{definition}

\newtheorem{prop}[theorem]{Proposition}

\newtheorem{cor}[theorem]{Corollary}
\newtheorem{rem}[theorem]{Remark}
\newtheorem{open}[theorem]{Open Question}

\theoremstyle{remark}
\newtheorem{remark}[theorem]{Remark}

\numberwithin{equation}{section}

\begin{document}

\title[Szeg\H{o} and Bergman Projections on Planar Domains]{Some Results for the Szeg\H{o} and Bergman Projections on Planar Domains}


\author{Nathan A. Wagner}
\address{Nathan A. Wagner \hfill\break\indent 
 Department of Mathematics \hfill\break\indent 
 Brown University \hfill\break\indent 
 151 Thayer Street\hfill\break\indent 
 Providence, RI, 02912 USA}
\email{nathan\_wagner@brown.edu}
\thanks{The research that led to this manuscript was partially completed when the author was supported by an NSF GRF grant (DGE-1745038).}


\date{today}

\begin{abstract}
The purpose of this note is to prove some boundedness/compactness results of a harmonic analysis flavor for the Bergman and Szeg\H{o} projections on certain classes of planar domains using conformal mappings. In particular, we prove weighted estimates for the projections, provide quantitative $L^p$ estimates and a specific example of such estimates on a domain with a sharp $p$ range, and show that the ``difference'' of the Bergman and Szeg\H{o} projections is compact at the endpoints $p=1, \infty$ for domains with sufficient smoothness. We also pose some open questions that naturally arise from our investigation. 
\end{abstract}

\maketitle

\section{Introduction}

Let $\mathbb{D}$ denote the unit disc and $\partial \mathbb{D}$ the unit circle. Let $A^2(\mathbb{D})$ denote the Bergman space of square-integrable holomorphic functions on the unit disc and $H^2(\partial \mathbb{D})$ denote the Hardy space on the unit circle, which is the closed subspace of $L^2(\partial \mathbb{D})$ consisting of functions whose Fourier series are of the form $f(e^{\ii \theta})= \sum_{n=0}^{\infty} a_n e^{\ii n \theta}.$ It is well-known that every function in $H^2(\partial \mathbb{D})$ can be identified with a holomorphic function in $\mathbb{D}$ with square-summable Taylor coefficients (equivalently, belonging to the space $H^2(\mathbb{D})$) via the Poisson extension, and conversely, the boundary values of any such holomorphic function are well-defined as non-tangential limits almost everywhere on $\partial \mathbb{D}$ and belong to $H^2(\partial \mathbb{D}).$ The Bergman projection $\mathcal{B}_0$ is the orthogonal projection from $L^2(\mathbb{D})$ to $A^2(\mathbb{D})$, and can be expressed as the following integral for $f \in L^2(\mathbb{D})$ and $z \in \mathbb{D}$:

$$ \mathcal{B}_0f(z)= \int_{\mathbb{D}} \frac{f(w)}{(1-\overline{w}z)^2}        \, dA(w),$$
where $dA$ denotes (normalized) Lebesgue area measure on $\mathbb{D}.$ On the other hand, the Szeg\H{o} projection $\mathcal{S}_0$ is the orthogonal projection from  $L^2( \partial \mathbb{D})$ to $H^2( \partial \mathbb{D})$, and has the following integral representation for $f \in L^2(\partial \mathbb{D})$ and $\zeta \in \partial \mathbb{D}$:

$$ \mathcal{S}_0f(\zeta)= \lim_{r \rightarrow 1^{-}} \int_{0}^{2 \pi} \frac{f(e^{\mathrm{i} \theta})}{(1-e^{-\mathrm{i} \theta} r \zeta)}        \, d\tilde{\theta},\quad d\tilde{\theta}=\frac{d \theta}{2\pi}.$$

Conformal mapping then allows us to express the Bergman and Szeg\H{o} projections on other planar domains in terms of the projections on the unit disc/circle. Let $\Omega \subsetneq \mathbb{C}$ be a simply connected domain. When we consider the Szeg\H{o} projection, we will make the additional assumption that $\Omega$ is bounded and has rectifiable boundary. For such a domain, we know by the Riemann mapping theorem that there exists a biholomorphic map $\varphi: \mathbb{D} \rightarrow \Omega.$ In particular, if $\Omega$ is also assumed to be bounded and rectifiable, we know that $\varphi$ extends to a continuous bijective mapping of the closures $\bar{\mathbb{D}} \rightarrow \bar{\Omega}$ and $\varphi'$ has non-tangential boundary limits almost everywhere. Moreover, if we identify $\varphi'$ with its boundary values, it is known $\varphi' \in H^1(\partial \mathbb{D})$ (see \cite{LS1}). When we refer to $\varphi'$ in the case of the Szeg\H{o} projection, we will make this identification without further comment.

  The Bergman projection $\mathcal{B}$ is again defined to be the orthogonal projection from $L^2(\Omega)$ to $A^2(\Omega)$, where $A^2(\Omega)$ is the Bergman space of square-integrable functions on $\Omega.$ Having fixed a conformal map $\varphi,$ define a map $\tau$ from $L^2(\Omega)$ to $L^2(\mathbb{D})$ as follows:  $$ \tau(f):=(f \circ \varphi )\cdot (\varphi')$$  
and notice that $\tau$ is a surjective isometry. Then, using a transformation law for the Bergman kernel, one can show the Bergman projection $\mathcal{B}$ can be expressed \begin{equation} \mathcal{B}=\tau^{-1} \mathcal{B}_0 \tau. \label{Bergman relation} \end{equation}

 Similarly, the Szeg\H{o} projection $\mathcal{S}$ can be defined as the orthogonal projection from $L^2(\partial \Omega)$ to $H^2(\partial \Omega)$, where $H^2(\partial \Omega)$ is an appropriate notion of Hardy space. This definition can be made more explicit by conformal mapping: analogously define a surjective isometry $\tau_{1/2}$ from $L^2(\partial \Omega)$ to $L^2(\partial \mathbb{D})$ as follows:

 $$\tau_{1/2}(f)(e^{\ii \theta}):= f(\varphi(e^{\ii \theta})) \cdot (\varphi'(e^{\ii \theta}))^{1/2}.$$
We remark that as a consequence of the fact that $\varphi'$ is non-vanishing on $\mathbb{D}$, there exists a holomorphic function $\phi$ on $\mathbb{D}$ such that $\phi^2=\varphi'$. Moreover, $\phi$ has non-tangential boundary limits almost everywhere (in fact, $\phi \in H^2(\mathbb{D})$) by virtue of the fact that $\varphi' \in H^1(\mathbb{D})$. Therefore, we may take $(\varphi'(e^{\ii \theta}))^{1/2}=\phi(e^{\ii \theta})$ for almost every $\theta \in [0,2 \pi)$, and this consistently defines a $L^2(\partial \mathbb{D})$ function. Then the Szeg\H{o} projection on $L^2(\partial \Omega)$ may be defined 
\begin{equation} \mathcal{S}= \tau_{1/2}^{-1} \mathcal{S}_0 \tau_{1/2}. \label{Szego relation} \end{equation}

While the $L^2$ regularity of both projection operators is immediate from their definitions, the $L^p$ regularity is highly dependent on the boundary geometry of the domain. The technique of relating the projections to operators on the disk via conformal mapping was used by Lanzani and Stein to study the $L^p$ regularity of the Bergman and Szeg\H{o} projections in \cite{LS1} on several classes of planar domains including graph domains, local graph domains, domains with the vanishing chord-arc property, and Lipschitz domains. In some of these cases (local graph, graph, and Lipschitz) they show that the Szeg\H{o}/Bergman projections are bounded on $L^p$ for a limited range $p_0<p<p_0'=\frac{p_0}{p_0-1}$ where $p_0>1$, and they additionally demonstrate that the range cannot be improved in the Lipschitz case for the Szeg\H{o} projection. This behavior of the projections stands in stark contrast to the smooth boundary case, or even the case when the boundary is of class $C^1$ or convex. In these situations, the $L^p$ boundedness for $1<p<\infty$ can be recovered (see, for example \cite[Theorem 2.1] {Bek1}, \cite{S,S2}). However, in the case that $\Omega$ has the vanishing chord-arc property, Lanzani and Stein proved that the projections are bounded on $L^p$ for the full range as well, which strengthens the $C^1$ result. 

  It turns out that the \emph{unweighted} $L^p$ regularity of either projection on $\Omega$ or $\partial \Omega$ is equivalent to the boundedness on a \emph{weighted} $L^p$ space of the corresponding projection on the unit disc/circle. The boundedness on the unit disc/circle can then be deduced from known weighted theory; in particular, the Szeg\H{o} projection is bounded on $L^p_\omega(\partial \mathbb{D})$ if and only if $\omega \in A_p(\partial \mathbb{D})$ (see \cite{Nik}), and the Bergman projection is bounded on $L^p_\omega(\mathbb{D})$ if and only if $\omega \in A_p^{+}(\mathbb{D})$ (see \cite{Bek1}).
These weight classes are defined precisely in Section \ref{Weighted}. We also mention that this conformal mapping technique was used by Lanzani and Koenig \cite{LK} to study mapping properties of the difference operator $\mathcal{B}\mathcal{P}-\mathcal{P} \mathcal{S},$ where $\mathcal{P}$ denotes the Poisson extension operator. 

In this note, we prove a collection of results, including weighted estimates and compactness, for the Bergman and Szeg\H{o} projections and related operators on planar domains. The unifying principle in our exploration is the use of conformal mapping techniques, in particular equations \eqref{Bergman relation} and \eqref{Szego relation} and related identities to prove these harmonic analysis theorems. First, we provide sufficient conditions on a weight $\omega$ for the Bergman or Szeg\H{o} projection to be bounded on $L^p_\omega$ for a general planar domain. In specific cases, these conditions translate to membership in an expected weight class that is the intersection of a Muckenhoupt class with a reverse H\"{o}lder class. We also use conformal mapping to provide a quantitative upper estimate on the operator norms of the projections on unweighted $L^p$ spaces and work out a specific example where we also provide a lower bound. Finally, we study the compactness of the difference operator $\mathcal{B}\mathcal{P}-\mathcal{P}\mathcal{S}.$ Throughout the paper, we list several open questions as well.

This paper is organized as follows. In Section \ref{Weighted}, we prove the weighted $L^p$ estimates for the projection operators (Theorems \ref{SzegoMain} and \ref{BergmanMain}) and note several applications to specific domain classes (Corollaries \ref{SzegoGraph}, \ref{SzegoLip}, \ref{BergmanGraph}, and \ref{BergmanLip}). In Section \ref{Quant}, we show how the conformal techniques lead to quantitative (unweighted) $L^p$ norm estimates for the projections and work out an example where we can obtain concrete upper and lower bounds (Theorem \ref{SzegoSpecificQuant}). Finally, in Section \ref{B vs S}, we prove endpoint compactness results for the difference operator $\mathcal{B} \mathcal{P}-\mathcal{P}\mathcal{S}$ (Theorem \ref{CompactDiffInf}).

The author would like to thank Brett Wick, Cody Stockdale, and Walton Green for several helpful comments. He would also like to express his gratitude to the anonymous referee, whose suggestions led to many improvements in the paper.

\section{Weighted Estimates}\label{Weighted}

In this section, we prove general weighted estimates for the Bergman and Szeg\H{o} projections. As a consequence, we deduce specific results on certain classes of planar domains. To summarize, we establish that for these domain classes, if the Szeg\H{o} or Bergman projection is bounded on $L^p$ for a limited range of $p$, then the operator extends boundedly in the same range to the weighted space $L^p_{\omega}(\Omega),$ provided the pullback of the weight $\omega$ to the unit disc belongs to an appropriate Muckenhoupt class intersected with a reverse H\"{o}lder class. In particular, see Subsection \ref{Apps}.

Weighted theory for operators that are bounded on a limited range of $L^p$ spaces has been studied by several authors. Auscher and Martell proved that for a certain class of operators $\mathcal{T}$ that are bounded on a limited range $p_0<p<q_0$ of $L^p$ spaces, weighted $L^p$ estimates can be obtained \cite{AM}. In particular, if $\omega \in A_{\frac{p}{p_0}} \cap \text{RH}_{\left(\frac{q_0}{p}\right)'}$, then $\mathcal{T}$ maps $L^p(\omega)$ to $L^p(\omega)$. Bernicot, Frey, and Petermichel later elaborated on this result to provide a sharp weighted estimate using techniques of sparse bounds \cite{BFP}. The weight classes we consider below are inspired by these results.

Recall that a weight is a locally integrable function that is positive almost everywhere. For $1<p<\infty$, we say a weight $\omega$ on $\partial \mathbb{D}$ belongs to $A_p(\partial \mathbb{D})$ if the following quantity is finite:
$$[\omega]_{A_p}:= \sup_{I} \frac{1}{|I|}\int_{I} \omega \, d \tilde{\theta} \left(\frac{1}{|I|}\int_{I} \omega^{-\frac{1}{p-1}}\, d \tilde{\theta}             \right)^{p-1},$$
where $I$ is an arc on the boundary of $\mathbb{D,}$ $d \tilde{\theta}$ denotes normalized arc length measure, and $|I|= \int_{I} \, d \tilde{\theta}$ denotes the (normalized) arc length of $I$. On the other hand, for $1<q<\infty$ we say that $\omega$ lies in the reverse H\"{o}lder class $\text{RH}_q(\partial \mathbb{D})$ if

$$[\omega]_{\text{RH}_q(\partial \mathbb{D})}:= \sup_{I} \frac{\left(\frac{1}{|I|} \int_{I} \omega^q \, d \tilde{\theta} \right)^{1/q}}{\frac{1}{|I|} \int_{I} \omega\, d \tilde{\theta}}< \infty.$$
Obviously, the normalization is not important in the definition of these weight classes; it is simply for convenience. 

Analogously for the Bergman projection, we define the $A_p^{+}$ class of weights with respect to certain Carleson regions. For $z_0 \in \mathbb{C}$, let $D_{r}(z_0)=\{z:|z-z_0|<r\}$  be the (Euclidean) disc of radius $r$ centered at $z_0$. We say a weight $\omega$ on $\mathbb{D}$ belongs to $A_p^{+}(\mathbb{D})$ if the following quantity is finite:

$$[\omega]_{A_p^{+}}:= \sup_{\substack{ z \in \partial \mathbb{D} \\ r>0}} \frac{1}{|D_r(z) \cap \mathbb{D}|}\int_{D_r(z)\cap \mathbb{D}} \omega \, dA \left(\frac{1}{|D_r(z) \cap \mathbb{D}|}\int_{D_r(z) \cap \mathbb{D}} \omega^{-\frac{1}{p-1}} \, dA \right)^{p-1},$$ where $dA$ denotes (normalized) planar Lebesgue measure and $|D_r(z)|=A(D_r(z)).$

Analogous to our earlier definition, for $1<q<\infty$ we say that $\omega \in \text{RH}_q(\mathbb{D})$ if 

$$[\omega]_{\text{RH}_q(\mathbb{D})}:= \sup_{\substack{ z \in \partial \mathbb{D} \\ r>0}} \frac{\left(\frac{1}{|D_r(z) \cap \mathbb{D}|} \int_{D_r(z) \cap \mathbb{D}} \omega^q \, dA \right)^{1/q}}{\frac{1}{|D_r(z) \cap \mathbb{D}|} \int_{D_r(z) \cap \mathbb{D}} \omega \, dA}< \infty.$$
We remark in passing that a subclass of $A_p^{+}$ weights (more precisely, a limiting class when $p=\infty$ that satisfy an additional regularity condition) that belong to $\text{RH}_q(\mathbb{D})$ were the main object of study in \cite{APR}.

We have the following two general theorems. 

\begin{theorem}\label{SzegoMain}
Suppose $\Omega \subset \mathbb{C}$ is a bounded, simply connected domain with rectifiable boundary and $1<p_0<2$. Suppose there exists a biholomorphic mapping $\varphi:\bar{\mathbb{D}} \rightarrow \bar{\Omega}$ that has the property that $|\varphi'|^{(1-\frac{p}{2})\left(\frac{p_0}{p(p_0-1)}\right)} \in A_2(\partial \mathbb{D})$ for $p_0<p<\frac{p_0}{p_0 -1}.$ If $p_0<p<\frac{p_0}{p_0-1}$ and $\omega \circ \varphi \in A_{\frac{p}{p_0}}(\partial\mathbb{D}) \cap \emph{RH}_{\left(\frac{p_0'}{p}\right)'}(\partial \mathbb{D})$, then the Szeg\H{o} projection $\mathcal{S}$ extends to a bounded operator on $L^p_\omega(\partial \Omega)$. Moreover,

$$\|\mathcal{S}\|_{L^p_\omega(\partial \Omega) \rightarrow L^p_\omega(\partial \Omega)} \lesssim C_p \left([\omega \circ \varphi]_{A_{\frac{p}{p_0}}} \times [\omega \circ \varphi]_{\emph{RH}_{\left(\frac{p_0'}{p}\right)'}} \times \left[|\varphi'|^{(1-\frac{p}{2})\left(\frac{p_0}{p(p_0-1)}\right)}\right]^{\frac{p(p_0-1)}{p_0}}_{A_2}\right)^{\max\{1, \frac{1}{p-1}\}},$$
where $C_p= \max\{p, \frac{1}{p-1}\}$ and the implicit constant is independent of $p$ and $\Omega.$

\end{theorem}

\begin{theorem}\label{BergmanMain}
Suppose $\Omega \subsetneq \mathbb{C}$ is a simply connected domain. Suppose there exists a biholomorphic mapping $\varphi:\mathbb{D} \rightarrow \Omega$ that has the property that $|\varphi'|^{(2-p)\left(\frac{p_0}{p(p_0-1)}\right)} \in A^+_2(\mathbb{D})$ for $p_0<p<\frac{p_0}{p_0 -1}.$ If $p_0<p<\frac{p_0}{p_0-1}$ and $\omega \circ \varphi \in A^+_{\frac{p}{p_0}}(\mathbb{D}) \cap \emph{RH}_{\left(\frac{p_0'}{p}\right)'}(\mathbb{D})$, then the Bergman projection $\mathcal{B}$ is bounded on $L^p_\omega(\Omega)$. Moreover,

$$\|\mathcal{B}\|_{L^p_\omega(\Omega) \rightarrow L^p_\omega(\Omega)} \lesssim C_p  \left([\omega \circ \varphi]_{A^+_{\frac{p}{p_0}}} \times [\omega \circ \varphi]_{\emph{RH}_{\left(\frac{p_0'}{p}\right)'}} \times \left[|\varphi'|^{(2-p)\left(\frac{p_0}{p(p_0-1)}\right)}\right]^{\frac{p(p_0-1)}{p_0}}_{A_2^+}\right)^{\max\{1, \frac{1}{p-1}\}},$$
where $C_p= \max\{p, \frac{1}{p-1}\}$ and the implicit constant is independent of $p$ and $\Omega.$

\end{theorem}

We denote arc length measure on $\partial \Omega$ by $\, d \sigma.$ For convenience in what follows, for a locally integrable function $f$  we let
$$\fint_{I} f \, d \tilde{\theta}= \frac{1}{|I|} \int_{I} f \, d \tilde{\theta}.$$
We use analogous notation for averages in the Bergman case when the integrals are taken over a domain in $\mathbb{C}$ with respect to Lebesgue measure $dA$ in the plane. 

\subsection{Proofs of main theorems}

These theorems are established by transferring the problem to the unit disc via the following two lemmas. In what follows, $\varphi: \bar{\mathbb{D}} \rightarrow \bar{\Omega}$ always denotes the (fixed) Riemann map.

\begin{lemma} \label{transfer to disc Szego}
Let $1 \leq p<\infty$. The boundedness of $\mathcal{S}$ on $L^p_\omega(\partial \Omega)$ is equivalent to the boundedness of $\mathcal{S}_0$ on $L^p_{(\omega \circ \varphi) \nu}(\partial \mathbb{D})$ where $\nu(e^{\ii \theta})=|\varphi'(e^{\ii \theta})|^{1-\frac{p}{2}}.$ Therefore, for $1<p<\infty$, the boundedness of $\mathcal{S}$ on $L^p_\omega(\partial \Omega)$ is equivalent to the condition $(\omega \circ \varphi)\nu \in A_p(\partial \mathbb{D})$.

\begin{proof}
 To prove the first statement, it suffices to show that $\tau_{1/2}$ is also a surjective isometry from $L^p_{\omega}(\partial \Omega)$ to $L^p_{(\omega \circ \varphi)\nu}(\partial \mathbb{D})$. It is well-known (see \cite[pp. 66]{LS1}) that $d \sigma(\zeta)= |\varphi'(e^{\ii \theta})|d\theta$ where $\zeta= \varphi(e^{\ii\theta})$, so we may apply a change of variable. We compute, for a function $f \in L^p_{\omega}(\partial \Omega)$:

\begin{eqnarray*}
\int_{\partial \Omega} |f|^p \omega \mathop{d \sigma}& = & \int_{\partial \mathbb{D}} |f(\varphi)|^p (\omega \circ \varphi) |\varphi'| \mathop{d \theta} \\
& = & \int_{\partial \mathbb{D}} |f(\varphi) |^p |\varphi'|^{p/2} (\omega \circ \varphi) |\varphi'|^{1-\frac{p}{2}} \mathop{d \theta}\\
& = &\int_{\partial \mathbb{D}} |\tau_{1/2}(f)|^p (\omega \circ \varphi) |\varphi'|^{1-\frac{p}{2}} \mathop{d \theta}
\end{eqnarray*}
which establishes that $\tau_{1/2}$ is isometric. Surjectivity is  obvious.

The second statement is an immediate consequence of the first by standard facts about the Szeg\H{o} projection on the unit disc (see \cite[pp. 66]{LS1} or \cite{Nik}).
\end{proof}

\end{lemma}

The second lemma is completely analogous to the first. 

\begin{lemma} \label{TransferDiscBergman}
The boundedness of $\mathcal{B}$ on $L^p_\omega( \Omega)$ is equivalent to the boundedness of $\mathcal{B}_0$ on $L^p_{(\omega \circ \varphi) \nu}( \mathbb{D})$ where $\nu(z)=|\varphi'(z)|^{2-p}.$ Therefore, the boundedness of $\mathcal{B}$ on $L^p_\omega(\Omega)$ is equivalent to the condition $(\omega \circ \varphi)\nu \in A_p^{+}( \mathbb{D})$.

\begin{proof}
The proof is virtually identical to the proof of the preceding lemma with minor changes. In particular, the map $\tau_{1/2}$ needs to be replaced by the map $\tau$ and the integration now takes places on $\mathbb{D}.$

\end{proof}
\end{lemma}

The following two propositions will enable us to prove Theorems \ref{SzegoMain} and \ref{BergmanMain}.

\begin{prop}\label{SzegoHolder}
Let $1<p_0<p<\frac{p_0}{p_0-1}$ and suppose $(\omega \circ \varphi) \in A_{\frac{p}{p_0}} (\partial \mathbb{D})\cap \text{RH}_{\left(\frac{p_0'}{p}\right)'}(\partial \mathbb{D}).$ Let $\nu(e^{\ii \theta})=|\varphi'(e^{\ii \theta})|^{1-\frac{p}{2}}.$ If $\nu^{\frac{p_0}{p(p_0-1)}} \in A_2(\partial \mathbb{D})$, then $(\omega \circ \varphi)\nu \in A_p(\partial \mathbb{D})$ and $$[(\omega \circ \varphi)\nu ]_{A_p} \leq [\omega \circ \varphi]_{A_{\frac{p}{p_0}}}\times  [\omega \circ \varphi]_{\text{RH}_{\left(\frac{p_0'}{p}\right)'}} \times \left[\nu^{\frac{p_0'}{p}}\right]_{A_2}^{\frac{p}{p_0'}}.$$

\begin{proof}
We must show 
$$\sup_{I} \fint_{I} (\omega \circ \varphi) \nu \mathop{d \tilde{\theta}} \left(\fint_{I} ((\omega \circ \varphi)\nu)^{-\frac{1}{p-1}}\mathop{d \tilde{\theta}}\right)^{p-1} \leq [\omega \circ \varphi]_{A_{\frac{p}{p_0}}}\times  [\omega \circ \varphi]_{\text{RH}_{\left(\frac{p_0'}{p}\right)'}} \times \left[\nu^{\frac{p_0'}{p}}\right]_{A_2}^{\frac{p}{p_0'}}.$$ 

Applying H\"older's inequality with exponents $r=\frac{p_0'}{p}$ and $r'$ and using the fact that $(\omega \circ \varphi) \in \text{RH}_{r'}(\partial \mathbb{D})$, we obtain
\begin{eqnarray*}
\fint_{I}(\omega \circ \varphi)\nu \mathop{d \tilde{\theta}} & \leq &  \left(\fint_{I} (\omega \circ \varphi)^{r'}\mathop{d \tilde{\theta}} \right)^{\frac{1}{r'}} \left(\fint_{I} \nu^r \mathop{d \tilde{\theta}}\right)^{\frac{1}{r}}\\
& \leq & [\omega \circ \varphi]_{\text{RH}_{r'}} \left(\fint_{I} (\omega \circ \varphi)\mathop{d \tilde{\theta}} \right) \left(\fint_{I} \nu^r \mathop{d \tilde{\theta}}\right)^{\frac{1}{r}}.
\end{eqnarray*}

On the other hand, let $s=\frac{p}{p_0}$. Let $t=\frac{(p-1)}{s-1}$ and note $t>1$. Then, applying H\"{o}lder's inequality with exponents $t$ and $t'=\frac{p_0(p-1)}{p(p_0-1)}=(p-1)r$, we obtain

\begin{eqnarray*}
\left(\fint_{I} ((\omega \circ \varphi)\nu)^{-\frac{1}{p-1}}\mathop{d \tilde{\theta}}\right)^{p-1} & \leq & \left(\fint_{I} (\omega \circ \varphi)^{-\frac{1}{s-1}} \mathop{d \tilde{\theta}}\right)^{s-1} \left(\fint_{I}\nu^{-r} \mathop{d \tilde{\theta}}\right)^{\frac{1}{r}}.
\end{eqnarray*}

Putting these facts together, we then get

\begin{align*}
& \fint_{I} (\omega \circ \varphi) \nu \mathop{d \tilde{\theta}} \left(\fint_{I} ((\omega \circ \varphi)\nu)^{-\frac{1}{p-1}}\mathop{d \tilde{\theta}}\right)^{p-1} \\
\leq & [\omega \circ \varphi]_{\text{RH}_{r'}} \left(\fint_{I} (\omega \circ \varphi)\mathop{d \tilde{\theta}} \right) \left(\fint_{I} \nu^{r} \mathop{d \tilde{\theta}}\right)^{\frac{1}{r}} \left(\fint_{I} (\omega \circ \varphi)^{-\frac{1}{s-1}} \mathop{d \tilde{\theta}}\right)^{s-1} \left(\fint_{I}\nu^{-r} \mathop{d \tilde{\theta}}\right)^{\frac{1}{r}}\\
\leq &   [\omega \circ \varphi]_{A_s(\partial \mathbb{D})} \times [\omega \circ \varphi]_{\text{RH}_{r'}} \times \left[\nu^{r}\right]_{A_2(b\mathbb{D})}^{\frac{1}{r}}
\end{align*}
as required. 

\end{proof}
\end{prop}

\begin{prop} \label{BergmanHolder}
Let $1<p_0<p<\frac{p_0}{p_0-1}$ and suppose $(\omega \circ \varphi) \in A^{+}_{\frac{p}{p_0}} (\mathbb{D})\cap \text{RH}_{\left(\frac{p_0'}{p}\right)'}(\mathbb{D}).$ Let $\nu(z)=|\varphi'(z)|^{2-p}.$ If $\nu^{\frac{p_0}{p(p_0-1)}} \in A^{+}_2(\mathbb{D})$, then $(\omega \circ \varphi)\nu \in A^{+}_p(\mathbb{D})$ and $$[(\omega \circ \varphi)\nu ]_{A_p^+} \leq [\omega \circ \varphi]_{A_{\frac{p}{p_0}}^+}\times  [\omega \circ \varphi]_{\text{RH}_{\left(\frac{p_0'}{p}\right)'}} \times \left[\nu^{\frac{p_0'}{p}}\right]_{A_2^+}^{\frac{p}{p_0'}}.$$
\end{prop}

\begin{proof}
The proof follows exactly the same way as Proposition \ref{SzegoHolder}.
\end{proof}

\begin{proof}[Proof of Theorem \ref{SzegoMain}]
The theorem will follow from Lemma \ref{transfer to disc Szego} and Proposition \ref{SzegoHolder} once we have deduced the following weighted norm estimates for the Szeg\H{o} projection on the circle: $$\|\mathcal{S}_0\|_{L^p_\omega( \partial \mathbb{D}) \rightarrow L^p_\omega( \partial \mathbb{D}) } \lesssim C_p [\omega]_{A_p(\partial \mathbb{D})}^{\max\{1,\frac{1}{p-1}\}}, \quad C_p= \max\{p, \frac{1}{p-1}\}.$$

This estimate can be deduced as follows. We will begin by proving it for $p=2$, and then we will extend the estimate to all $p \in (1,\infty)$ via extrapolation. It is also probably possible in some way to deduce this bound from the case of the Hilbert transform on the line, but we argue directly instead.  First, note that the Szeg\H{o} projection on the circle has kernel (with respect to normalized arc length measure $d \tilde{\theta}$) $$K(z,w)=\frac{1}{1-\overline{w}z}$$ in the precise sense that if $f \in L^2(\partial \mathbb{D})$ we have, for almost every  $z \in \partial{\mathbb{D}} \setminus \text{supp}{f}$:
$$ \mathcal{S}_0 f(z)= \int_{\partial \mathbb{D}} \frac{f(w)}{1-\overline{w}z}\,  d \tilde{\theta}(w). $$
This fact is a well-known consequence of the Plemelj jump formula. 

Moreover, we see that $\mathcal{S}_0$ can be viewed as a Calder\'{o}n-Zygmund operator on the circle $\partial \mathbb{D}$ viewed as a space of homogeneous type with respect to the measure $d\tilde{\theta}$ and arc length metric $$d(z,w)=|\theta-\phi|, \quad z=e^{\ii \theta}, w= e^{\ii \phi} \text{ and }\theta, \phi \in [0,2\pi).$$ We remark in passing that one could also take the Euclidean metric by viewing $z,w$ as living in the ambient space $\mathbb{C}$, and this would have no serious impact on the proof. Then, $\mathcal{S}_0$ is obviously bounded on $L^2( \partial \mathbb{D})$, and the size and smoothness estimates on its kernel then follow from direct calculation, which we omit. The sharp weighted estimates on $L^p$ for $p=2$ then follow from the $A_2$ theorem on spaces of homogeneous type (see \cite[Theorem 1.1]{AV}), which states that if $T$ is a Calder\'{o}n-Zygmund operator on a space of homogeneous type and $w$ is an $A_2$ weight, then

\begin{equation} \|T\|_{L^2(w) \rightarrow L^2(w)} \lesssim [w]_{A_2} .\label{A2Theorem} \end{equation}

Given $\zeta \in \partial \mathbb{D}$ and $0<r<2 \pi$, let $I(\zeta,r)= \{\xi \in \partial \mathbb{D}: d(\xi,\zeta)<r\}$. Let $\mathfrak{M}_{\partial \mathbb{D}}$ be the Hardy-Littlewood Maximal function defined 
$$\mathfrak{M}_{\partial \mathbb{D}}f(z)= \sup_{\substack{\zeta \in \partial \mathbb{D}\\r \in (0,2 \pi):\\ z \ni I(\zeta,r) }} \frac{1}{|I(\zeta,r)|} \int_{I(\zeta,r)}|f| \, d\tilde{\theta}. $$

A line-by-line imitation of the Euclidean extrapolation proof (replace $\mathbb{R}^n$ by $\partial \mathbb{D}$, etc) in \cite{CMP}*{Theorem 3.22}, together with \eqref{A2Theorem} yields the estimates 

\begin{equation} \int_{\partial \mathbb{D}} |\mathcal{S}_0f| ^p \omega \, d\theta\lesssim \|\mathfrak{M}_{\partial \mathbb{D}}\|_{L^p_{\omega}(\partial \mathbb{D}) \rightarrow L^p_{\omega}(\partial \mathbb{D})}^{2-p}  [\omega]_{A_p} \left(\int_{\partial \mathbb{D}} |f| ^p \omega \, d \theta \right)  ,\quad 1<p<2 \label{ExtrapLower}\end{equation}

and

\begin{equation} \int_{\partial \mathbb{D}} |\mathcal{S}_0f| ^p \omega \, d\theta \lesssim \|\mathfrak{M}_{\partial \mathbb{D}}\|_{L^p_{\omega'}(\partial \mathbb{D}) \rightarrow L^p_{\omega'}(\partial \mathbb{D})}^{\frac{p-2}{p-1}}  [\omega]_{A_p}^{\frac{1}{p-1}} \left(\int_{\partial \mathbb{D}} |f| ^p \omega \, d\mu \right)  ,\quad \omega'= \omega^{1-p'} , \quad 2<p< \infty \label{ExtrapUpper} ,\end{equation}
where $\frac{1}{p}+\frac{1}{p'}=1$ and the implicit constants are independent of the exponent $p$ and weight $\omega.$ 

Finally, the following estimate for the maximal function (which is also valid on any space of homogeneous type) is given in \cite{HK}*{Proposition 7.13}:
\begin{equation}
\|\mathfrak{M}_{\partial \mathbb{D}}\|_{L^p_{\omega}(\partial \mathbb{D}) \rightarrow L^p_{\omega}(\partial \mathbb{D})} \lesssim (p')^{1/p} p^{1/p'} [\omega]_{A_p}^{\frac{1}{p-1}}, 1<p<\infty. \label{MaximalExtrap}
\end{equation}
The result then follows by combining \eqref{ExtrapLower}, \eqref{ExtrapUpper}, and \eqref{MaximalExtrap}.

We also refer the reader to \cite{DGP} for a similar estimate in the Euclidean setting of $\mathbb{R}^n.$ 

\end{proof}

\begin{proof}[Proof of Theorem \ref{BergmanMain}]
The theorem follows from Lemma \ref{TransferDiscBergman}, Proposition \ref{BergmanHolder}, and the fact that $\|\mathcal{B}_0\|_{L^p_\omega( \mathbb{D})} \lesssim C_p [\omega]_{A_p^+(\mathbb{D})}^{\max\{1,\frac{1}{p-1}\}}.$ This last fact can be found in \cite[Theorem 2]{RTW} (the quantative dependence on $p$ is implicit in the proof but is also stated explicitly in \cite[Theorem 1.2]{HWW}). In the notation of the reference, we take $b=0$ and $u=\omega$, $P_0=\mathcal{B}_0$, and $B_{p,0}= A_p^{+}(\mathbb{D}).$  We remark that here the quantitative estimate is a direct consequence of a ``sparse domination'' for the Bergman projection, and does not involve singular integral theory. Indeed, the Bergman projection is easier to deal with than the Szeg\H{o} projection because it is not a true singular integtal operator; cancellation does not play a role in its boundedness and only the size of its kernel matters. See also \cite{PR} for a similar result on the upper-half plane. 

\end{proof}

\subsection{Applications} \label{Apps}
We now present several applications of Theorems \ref{SzegoMain} and \ref{BergmanMain}. We consider in particular local graph domains and bounded Lipschitz domains. We say $\Omega$ is a graph domain if there exists an angle $\alpha>0$ and a continuous function $\gamma: \mathbb{R} \rightarrow \mathbb{R}$  so that if $z=x+\ii y,$ $$\Omega=\{e^{\ii \alpha}z: y > \gamma(x)\}.$$ Similarly, we say $\Omega$ is a local graph domain if there exist finitely many graph domains $\Omega_1,\dots, \Omega_n$ together with arcs on the boundary $\Gamma_1,\dots, \Gamma_n$ covering $\partial \Omega$ and discs $D_j$ so that $\Gamma_j \subset D_j$ and $\Omega \cap D_j= \Omega_j \cap D_j$ for $1 \leq j \leq n$. If each $\gamma_j$ is a Lipschitz map with Lipschitz constant $M_j$, we further say that $\Omega$ is a bounded Lipschitz domain with Lipschitz constant $M= \max_{1 \leq j \leq n} M_j.$
 
\begin{cor}\label{SzegoGraph} Let $\Omega$ be a bounded, simply connected local graph domain. Then  the Szeg\H{o} projection $\mathcal{S}$ is bounded on $L^p_\omega(\partial \Omega)$ for  $\frac{4}{3}<p<4$ and all $\omega$ satisfying $(\omega \circ \varphi) \in A_{\frac{3p}{4}}(\partial \mathbb{D}) \cap \text{RH}_{\left(\frac{4}{p}\right)'}(\partial \mathbb{D})$.
\end{cor}

\begin{cor} \label{SzegoLip} Let $\Omega$ be a bounded Lipschitz domain with Lipschitz constant $M$ and let $p_M=2(\frac{\pi}{2 \arctan{M}}+1).$ Then the Szeg\H{o} projection $\mathcal{S}$ is bounded on $L^p_\omega(\partial \Omega)$ for $p_M'<p<p_M$ and all $\omega$ satisfying $(\omega \circ \varphi) \in A_{\frac{p}{p_M'}}(\partial \mathbb{D}) \cap \text{RH}_{\left(\frac{p_M}{p}\right)'}(\partial \mathbb{D})$.
\end{cor}

\begin{cor}\label{BergmanGraph} Let $\Omega$ be a bounded, simply connected local graph domain. Then the Bergman projection $\mathcal{B}$ is bounded on $L^p_\omega( \Omega)$ for $\frac{4}{3}<p<4$ and all $\omega$ satisfying $(\omega \circ \varphi) \in A^{+}_{\frac{3p}{4}}( \mathbb{D}) \cap \text{RH}_{\left(\frac{4}{p}\right)'}(\mathbb{D})$.
\end{cor}

\begin{cor}\label{BergmanLip} Let $\Omega$ be a bounded Lipschitz domain with Lipschitz constant $M$ and let $p_M=2(\frac{\pi}{2 \arctan{M}}+1)$. Then the Bergman projection $\mathcal{B}$ is bounded on $L^p_\omega( \Omega)$ for $p_M'<p<p_M$ and all $\omega$ satisfying $(\omega \circ \varphi) \in A^{+}_{\frac{p}{p_M'}}(\mathbb{D}) \cap \text{RH}_{\left(\frac{p_M}{p}\right)'}(\mathbb{D})$.
\end{cor}

All of these corollaries are deduced via Theorem \ref{SzegoMain} or \ref{BergmanMain} in the same way; by verifying that $|\varphi'|^{(1-\frac{p}{2})(\frac{p_0}{p(p_0-1)})}$ (or its square) is an $A_2(\partial \mathbb{D})$ or (respectively, $A_2^{+}(\mathbb{D})$) weight. The details can be found in \cite{LS1}, but we recall relevant facts for completeness. If $\Omega$ is any local graph domain, it is shown in \cite[Proposition 2.2]{LS1} that $|\varphi'|^{r} \in A_2(\partial \mathbb{D})$ for $0\leq r <1$ (and thus actually for $-1 <r<1$). Note if $p_0=\frac{4}{3}$ and $p \in (\frac{4}{3},4)$, we have $(1-\frac{p}{2})(\frac{p_0}{p(p_0-1)})=\frac{4}{p}-2$, and $-1 < \frac{4}{p}-2 < 1$ for $p$ in this range. The argument for when $\Omega$ is a bounded Lipschitz domain is basically identical and follows from \cite[Corollary 2.3]{LS1} . 

On the other hand, for the same domains when considering the Bergman projection, Lanzani and Stein show that $\varphi'$ is an outer function on $\mathbb{D}$ \cite[Proposition 4.1]{LS1}. A result of B\'{e}koll\`{e} \cite[Theorem 1.2]{Bek1}\ is then invoked which states if $\varphi'$ is outer, $\alpha \in \mathbb{R}$, and $|\varphi'(e^{\ii \theta})|^{\alpha} \in A_p(\partial \mathbb{D})$, then $|\varphi'(z)|^{2 \alpha} \in A_p^{+}(\mathbb{D})$. Thus, the two corresponding results for the Bergman projection follow from the arguments above pertaining to the Szeg\H{o} projection.

\begin{rem}
Notice that the conditions on the weight for boundedness in all of these results are, of course, in terms of its pullback to the unit disc. It is unclear whether these weights can be described in a conformal-free way that is intrinsic to the domain $\Omega$, especially in the case of a domain with non-smooth boundary.
\end{rem}

\section{An Upper Bound and A Quantitative Example}\label{Quant}
\subsection{Upper Quantitative Bound}
One can use similar methods as above to obtain quantitative estimates on the norm of the unweighted Szeg\H{o} or Bergman projection. We have the following result:

\begin{theorem} \label{SzegoQuant}
Let $\Omega \subset \mathbb{C}$ be a simply connected, bounded domain with rectifiable boundary. Let $\varphi: \bar{\mathbb{D}} \rightarrow \bar{\Omega}$ be a conformal mapping. The following estimate holds for the Szeg\H{o} projection $\mathcal{S}$:
$$ \| \mathcal{S}\|_{L^p(\partial \Omega) \rightarrow L^p(\partial \Omega)  } \lesssim C_p \left[|\varphi'|^{1-\frac{p}{2}}\right]_{A_p}^{\max\{1,\frac{1}{p-1}\}},$$
where $C_p= \max\{p, \frac{1}{p-1}\}$ and the implicit constant is independent of $p$ and $\Omega.$
\end{theorem}

\begin{proof}[Proof of Theorem \ref{SzegoQuant}]
 First, note that the proof of Lemma \ref{transfer to disc Szego} demonstrates that $$\|\mathcal{S}\|_{L^p(\partial \Omega) \rightarrow L^p(\partial \Omega)}=\|\mathcal{S}_0\|_{L^p_{\omega}(\partial \mathbb{D}) \rightarrow L^p_{\omega}(\partial \mathbb{D})},$$ where $\omega=|\varphi'|^{1-\frac{p}{2}}.$ The upper bound then follows from noting that the Szeg\H{o} projection on the unit circle is a Calder\'{o}n-Zygmund operator on a space of homogeneous type as we argued in the proof of Theorem \ref{SzegoMain}. 

\end{proof}

Of course, we have an analogous statement for the Bergman projection, and in fact we can additionally provide a lower estimate in this case. 

\begin{theorem} \label{BergmanQuant}
Let $\Omega \subsetneq \mathbb{C}$ be a simply connected domain. Let $\varphi: \bar{\mathbb{D}} \rightarrow \bar{\Omega}$ be a conformal mapping. The following estimate holds for the Bergman projection $\mathcal{B}$:

$$ \left[|\varphi'|^{2-p}\right]_{A_p^+}^{\frac{1}{2p}} \lesssim \| \mathcal{B}\|_{L^p( \Omega) \rightarrow L^p( \Omega)} \lesssim C_p \left[|\varphi'|^{2-p}\right]_{A_p^+}^{\max\{1,\frac{1}{p-1}\}},$$
where $C_p= \max\{p, \frac{1}{p-1}\}$ and the implicit constants are independent of $p$ and $\Omega.$ 
\begin{proof}
Lemma \ref{TransferDiscBergman} gives that $$\|\mathcal{B}\|_{L^p( \Omega) \rightarrow L^p( \Omega)}=\|\mathcal{B}_0\|_{L^p_{\omega}(\mathbb{D})},$$ where $\omega=|\varphi'|^{2-p}.$ The upper bound then follows from \cite{RTW}*{Theorem 2} or \cite{HWW}*{Theorem 1.2} as before, while the lower bound can be found in \cite[Theorem 1.3]{HWW}.
\end{proof} 
\end{theorem}

\subsection{An Example} \label{example}
We now provide an example of a domain in the plane where the upper bound of the $L^p$ operator norm of the Szeg\H{o} projection can be estimated directly, and in fact a lower bound may be computed as well. This example is inspired by a family of examples in \cite{LS1}. Let $\mathbb{R}_{+}^{2}$ denote the upper half-plane in $\mathbb{C}.$ Consider the following planar domains:

\begin{align*}
\Omega_1= \Phi_1(\mathbb{R}_{+}^{2}), \Phi_1(z)=z^{3/2} e^{\frac{-\ii \pi}{4}} -\ii; & \hspace{0.4 cm} \Omega_2=\Phi_2(\mathbb{R}_{+}^{2}), \Phi_2(z)=-z^{3/2} e^{\frac{-\ii \pi}{4}}+\ii ; \\
\Omega_3=\Phi_3(\mathbb{R}_{+}^{2}), \Phi_3(z)=z^{1/2} e^{\frac{- \ii \pi}{4}}-2 ; & \hspace{0.4 cm} \Omega_4=\Phi_4(\mathbb{R}_{+}^{2}), \Phi_4(z)=-z^{1/2} e^{\frac{- \ii \pi}{4}}+2 ;\\
\Omega_5= \Phi_5(\mathbb{R}_{+}^{2}), \Phi_5(z)=z^{1/2} e^{\frac{ \ii \pi}{4}}-\left(\frac{1}{2}+\frac{3}{2}\ii\right); & \hspace{0.4 cm}
\Omega_6= \Phi_6(\mathbb{R}_{+}^{2}), \Phi_6(z)=z^{1/2} e^{\frac{ \ii \pi}{4}}+\left(\frac{1}{2}-\frac{3}{2}\ii\right);\\
\Omega_7= \Phi_7(\mathbb{R}_{+}^{2}), \Phi_7(z)=-z^{1/2} e^{\frac{ \ii \pi}{4}}+\left(-\frac{1}{2}+\frac{3}{2}\ii\right); & \hspace{0.4 cm}
\Omega_8= \Phi_8(\mathbb{R}_{+}^{2}), \Phi_8(z)=-z^{1/2} e^{\frac{ \ii \pi}{4}}+\left(\frac{1}{2}+\frac{3}{2}\ii\right).
\end{align*}

Then, let $\Omega= \bigcap_{j=1}^{8} \Omega_j$ (this is the same domain as $\Omega= \bigcap_{j=1}^{4} \Omega_j$, but the eight domains are needed to describe the boundary of $\Omega$ locally as graphs of functions). It is clear $\Omega$ is a local graph domain with corresponding graph domains $\Omega_j$, $1 \leq j \leq 8.$ We provide a picture of this domain created using Mathematica below (we note the green color corresponds to $\partial \Omega_3$, the blue corresponds to $\partial \Omega_1$, the red corresponds to $\partial \Omega_4$, and the orange corresponds to $\partial \Omega_2$):

\vspace{0.4 cm}

\begin{center}

\includegraphics{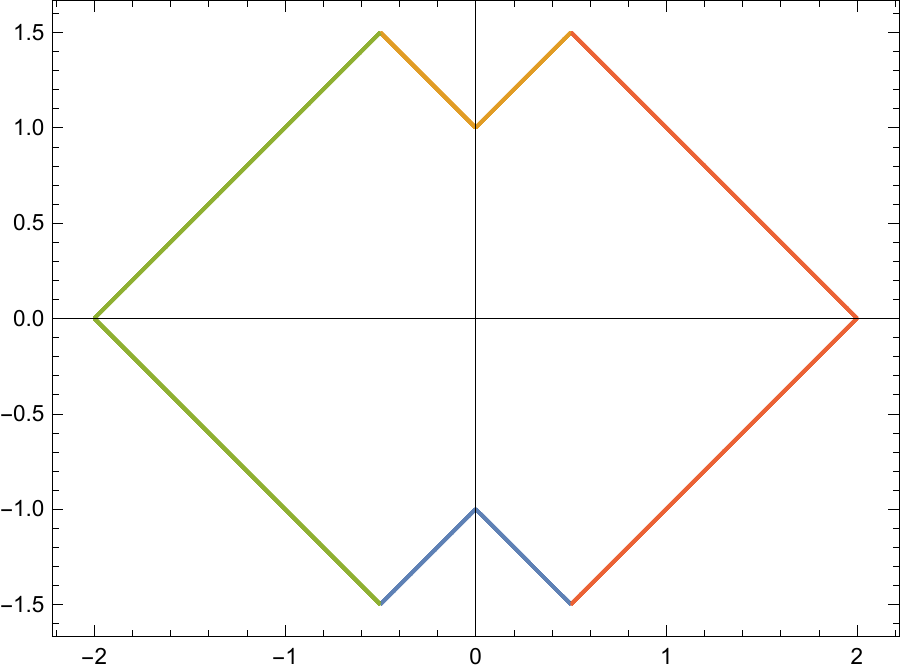}

\end{center}

\vspace{0.4 cm}

It is obvious that $\Omega$ is a Lipschitz domain and it is easy to see that the Lipschitz constant is $1$. We note that the points $(0,1)$ and $(0,-1)$ in $\mathbb{R}^2= \mathbb{C}$ are distinguished because the domain is not locally convex at those points. We know from previous work that since $M=1$, $\mathcal{S}$ is bounded on $L^p(\Omega)$ for $\frac{6}{5}<p<6$. The arguments in \cite{LS1} establish that $\mathcal{S}$ is in fact bounded only in this p range, but the authors do not provide quantitative estimates on the $L^p(\Omega)$ operator norm of $\mathcal{S}.$ Here we provide such estimates for the operator norm. 

Let $\varphi: \bar{\mathbb{D}} \rightarrow \bar{\Omega}$ be a conformal map. We wish to work via the conformal maps on the upper half space defined above. In particular, we remark that a conformal map $\Psi:\mathbb{R}_{+}^2 \rightarrow \Omega $ which fixes the point at infinity has non-tangential limits on the boundary of $\mathbb{R}_{+}^{2}$ almost everywhere. We now proceed to compute the $A_p(\partial \mathbb{D})$ characteristic of $|\varphi'|^{1-\frac{p}{2}}$ using the following proposition, which is a quantitative sharpening of \cite{LS1}*{Proposition 2.4}.
\begin{prop}\label{upper half plane to disc}
Let $\Omega$ be a local graph domain with corresponding graph domains $\Omega_1,\dots, \Omega_n$ and conformal maps $\Phi_j: \mathbb{R}_{+}^{2} \rightarrow \Omega_j$ which fix the point at $\infty.$ Let $1<p< \infty$ and suppose $|\Phi_j'(x)|^r \in A_p( \mathbb{R})$ for some $r>0$ and all $1 \leq j \leq n$. Let $\varphi:\bar{\mathbb{D}} \rightarrow \bar {\Omega}$ be a conformal map. Then $|\varphi'(e^{\ii \theta})|^r\in A_p(\partial \mathbb{D})$, and moreover:

$$ \left[ |\varphi'|^r\right]_{A_p(\partial \mathbb{D})} \lesssim C_\Omega^p \max\{N_1,N_2 \}$$
where $$N_1= \max_{1 \leq j \leq n} \left[|\Phi'_j|^r\right]_{A_p(\mathbb{R})}$$ and
$$N_2= \int_{\partial \mathbb{D} } |\varphi'|^r \mathop{d \tilde{\theta}} \left(\int_{ \partial \mathbb{D}} |\varphi'|^{-\frac{r}{p-1}}\mathop{d \tilde{\theta}} \right)^{p-1}.$$
Here $C_\Omega$ is a constant that only depends on $\Omega$.
\begin{proof}
The proof is essentially contained in \cite{LS1}, so we omit some of the details. Let $\psi_j= \varphi^{-1} \circ \Phi_j$. The main gist of the argument is that it is possible to choose a finite covering $I_1,\dots, I_n$ of $\partial \mathbb{D}$ together with intervals on $\mathbb{R}$ (identified as the boundary of $\mathbb{R}_{+}^{2}$), $J_j=\psi_j^{-1}(I_j)$, so that $|\psi'_j|$ is bounded above and below in a neighborhood of $ J_j$ and consequently $|\Phi_j'(z)| \approx |\varphi'(\psi_j(z))|$ for $z \in J_j$. Since there are finitely many such intervals $I_j$ and $|\Phi'_j|^r \in A_p(\mathbb{R})$ for all $j$, it is clear that $|\varphi'|^r$ and $|\varphi'|^{\frac{-r}{p-1}}$ belong to $L^1(\partial \mathbb{D})$. 

It is possible to choose $\delta>0$ small enough so that if $|I|< \delta$, then $I \subseteq I_j$ for some $j \in \{1,\cdots, n\}$ (this is just the Lebesgue number lemma applied to $\{I_j\}$). The parameter $\delta$ will depend on the underlying domain $\Omega.$ We then can estimate using change of variables:
\begin{align*}
\frac{1}{|I|}\int_{I} |\varphi'|^r \mathop{d \tilde{\theta}}\left(\frac{1}{|I|}\int_{I} |\varphi'|^{-\frac{r}{p-1}}\mathop{d \tilde{\theta}} \right)^{p-1} &
\approx \frac{1}{|\psi_j^{-1}(I)|}\int_{\psi_j^{-1}(I)} |\varphi' \circ \psi_j|^r |\psi_j'| \mathop{d x}\left(\frac{1}{|\psi_j^{-1}(I)|}\int_{\psi_j^{-1}(I)} |\varphi' \circ \psi_j|^{-\frac{r}{p-1}}|\psi_j'|\mathop{d x} \right)^{p-1}\\
& \approx \frac{1}{|\psi_j^{-1}(I)|}\int_{\psi_j^{-1}(I)} |\Phi_j'|^r \mathop{d x}\left(\frac{1}{|\psi_j^{-1}(I)|}\int_{\psi_j^{-1}(I)} |\Phi_j'|^{-\frac{r}{p-1}}\mathop{d x} \right)^{p-1}\\
& \leq N_1.
\end{align*}
All the implicit constants above evidently only depend on $\Omega.$

If $|I| \geq \delta$, then we estimate as follows:
\begin{align*}
\frac{1}{|I|}\int_{I} |\varphi'|^r \mathop{d \tilde{\theta}}\left(\frac{1}{|I|}\int_{I} |\varphi'|^{-\frac{r}{p-1}}\mathop{d \tilde{\theta}} \right)^{p-1} & \leq \frac{1}{\delta^p}\int_{\partial \mathbb{D}} |\varphi'|^r \mathop{d \theta}\left(\int_{\partial \mathbb{D}} |\varphi'|^{-\frac{r}{p-1}}\mathop{d \theta} \right)^{p-1} \\
& \leq \left(\frac{1}{\delta}\right)^p N_2.
\end{align*}
This completes the proof.

\end{proof}
\end{prop}
We now compute the $A_p(\mathbb{R})$ characteristics of $|\Phi_1'(x)|^{1-\frac{p}{2}}=\frac{3}{2}|x|^{\frac{2-p}{4}}$ and $|\Phi_3'(x)|=\frac{1}{2}|x|^{\frac{p-2}{4}},$ which are the relevant derivatives for the domain $\Omega$ introduced at the beginning of the section. The computations for the other derivatives are obviously identical so we omit them.

First, we compute the following average for $I$ an interval centered at $0$ with length $\delta$:
\begin{eqnarray*}
\frac{1}{|I|}\int_{I} |x|^{\frac{2-p}{4}} \mathop{d x} & \approx & \frac{1}{\delta} \int_{0}^{\delta} x^{\frac{2-p}{4}} \mathop{d x}\\
& = & \delta^{\frac{2-p}{4}} \left(\frac{4}{6-p}\right)
\end{eqnarray*}
and on the other hand
\begin{eqnarray*}
\left( \frac{1}{|I|}\int_{I} |x|^{\frac{p-2}{4(p-1)}} \mathop{d x} \right)^{p-1} & \approx & \left(\frac{1}{\delta} \int_{0}^{\delta} x^{\frac{p-2}{4(p-1)}} \mathop{d x}\right)^{p-1}\\
& = & \delta^{\frac{p-2}{4}} \left(\frac{4(p-1)}{5p-6}\right)^{p-1}.
\end{eqnarray*}
The product of these two factors gives
$$\left(\frac{4}{6-p}\right) \left(\frac{4(p-1)}{5p-6}\right)^{p-1}.$$

If $I$ is not centered at the origin, we consider two cases. Let $x_0$ denote the center of $I$ and $R$ the radius. First suppose $|x_0|<2R$. Then $I \subset 3I'$, where $I'$ is the translate of $I$ centered at $0$. Note $I$ and $3I'$ have comparable measure, so we can just apply the computation above to deduce

$$\left(\frac{1}{|I|}\int_{I} |x|^{\frac{2-p}{4}} \mathop{d x}\right)\left( \frac{1}{|I|}\int_{I} |x|^{\frac{p-2}{4(p-1)}} \mathop{d x} \right)^{p-1} \lesssim \left( \frac{4}{6-p} \right) \left(\frac{4(p-1)}{5p-6}\right)^{p-1}.$$
In the other case, $I$ is such that  $|x_0| \geq 2R$. In this case, if $x \in I$, $|x| \approx R$. It is then clear

$$\left(\frac{1}{|I|}\int_{I} |x|^{\frac{2-p}{4}} \mathop{d x}\right)\left( \frac{1}{|I|}\int_{I} |x|^{\frac{p-2}{4(p-1)}} \mathop{d x} \right)^{p-1} \approx 1.$$
 
It follows from these calculations that $$\left[|\Phi_1'|^{1-\frac{p}{2}}\right]_{A_p(\mathbb{R})} \approx \left(\frac{1}{6-p}\right) \left(\frac{4(p-1)}{5p-6}\right)^{p-1}.$$ Similar calculations may be employed to deduce $$\left[|\Phi_3'|^{1-\frac{p}{2}}\right]_{A_p(\mathbb{R})} \approx \left( \frac{1}{p+2} \right) \left(\frac{4(p-1)}{3p-2}\right)^{p-1}.$$ Note that the first characteristic blows up at the endpoints $p=\frac{6}{5}$ and $p=6$ while the second one does not. In fact, it is easy to see that there exists $C>0$ so that  

\begin{equation} \sup_{I} \fint |\Phi_3'|^{1-p/2}\, dx \leq C; \quad \sup_{I} \left(\fint |\Phi_3'|^{\frac{p/2-1}{p-1}}\, dx\right)^{p-1} \leq C \label{FactorsUniform} \end{equation}

The crucial behavior of the weight evidently occurs at the points $p_1=\varphi^{-1}(\ii)$ and $p_2=\varphi^{-1}(-\ii)$.  Thus, we may conclude using \eqref{FactorsUniform} together with the arguments in Proposition \ref{upper half plane to disc} that $$\left[ |\varphi'|^{1-\frac{p}{2}}\right]_{A_p(\partial \mathbb{D})} \approx \left( \frac{1}{6-p} \right)\left(\frac{4(p-1)}{5p-6}\right)^{p-1}.$$ 
Note that is it not difficult to see the lower bound by testing on small intervals centered at $p_1$ or $p_2.$

With this information, we can obtain upper and lower bounds the norm of the Szeg\H{o} projection in this case. First, we will need a preliminary lemma. 

\begin{lemma} \label{ball separation}
Suppose $r< \frac{2}{5}$. If $I_1=I_1(p_1,r)$ is an arc centered at $p_1 \in \partial \mathbb{D}$, then there exists a point $p_2 \in \partial \mathbb{D}$ and a disjoint arc $I_2=I_2(p_2,r)$ centered at $p_2$ so that for for $j \in \{1,2\}$ and any function $f>0$ supported in $I_j$ and $z \in I_k$ with $j \neq k$ we have the estimate:
$$|\mathcal{S}_0f(z)| \geq \frac{1}{28 \pi } \fint_{I_j} f \, d \tilde{\theta}.$$

\begin{proof}
Let $S_0(z,w)=\frac{1}{1-\overline{w}z}$ denote the kernel of $\mathcal{S}_0$ with respect to the measure $d \tilde{\theta}.$ An easy computation shows that $S_0$ satisfies the smoothness estimate
\begin{equation} |S_0(z,w)-S_0(z,w')| \leq \frac{3}{2} \frac{|w-w'|}{|w-z|^2}
\label{SzegoSmoothness} \end{equation}
 provided $|w-w'| \leq \frac{1}{3} |z-w|$. It suffices to prove the inequality for $j=1.$ Now, let  $w'= p_1$ as above. Select $p_2$ so that if $z \in I_2$ and $w \in I_1$, then $ 7r \geq |z-w| \geq 3 r \geq 3|w-w'|$ (this is possible if $r< \frac{2}{5}$). Then we compute, using \eqref{SzegoSmoothness} and noting that $z \notin \text{supp}(f)$:
\begin{eqnarray*}
|\mathcal{S}_0 f(z)| & = & \left| \int_{I_1} S_0(z,w) f(w) \mathop{d \tilde{\theta}(w)}\right|\\
& \geq& \left| \int_{I_1} S_0(z,w') f(w) \mathop{d \tilde{\theta}(w)} \right| - \left| \int_{I_1}[S_0(z,w)-S_0(z,w')] f(w) \mathop{d \tilde{\theta}(w)} \right| \\
& \geq & \frac{1}{|z-w'|}\int_{I_1} f(w) \mathop{d \tilde{\theta}(w)} - \int_{I_1}|S_0(z,w)-S_0(z,w')| f(w) \mathop{d \tilde{\theta}(w)}\\
& \geq & \frac{3}{4|z-w|}\int_{I_1} f(w) \mathop{d \tilde{\theta}(w)} - \int_{I_1}\frac{3}{2}  \frac{|w-w'|}{|w-z|^2} f(w) \mathop{d \tilde{\theta}(w)}\\
& \geq & \frac{1}{4 |z-w|}\int_{I_1} f(w) \mathop{d \tilde{\theta}(w)}\\
& \geq & \frac{1}{28 \pi} \fint_{I_1} f \, d \tilde{\theta},
\end{eqnarray*}
as required.

\end{proof}

\end{lemma}

\begin{theorem}\label{SzegoSpecificQuant}
The following estimate holds for the domain $\Omega$ defined at the beginning of Subsection \ref{example}:

$$\left(\frac{1}{6-p}\left(\frac{4(p-1)}{5p-6}\right)^{p-1}\right)^{\frac{1}{2p}} \lesssim \|\mathcal{S}\|_{L^p(\Omega) \rightarrow L^p(\Omega)} \lesssim \left(\frac{1}{6-p}\left(\frac{4(p-1)}{5p-6}\right)^{p-1}\right)^{\max\{1,\frac{1}{p-1}\}}.$$

\begin{proof}
The upper estimate follows from Theorem \ref{SzegoQuant} and the above computation of $\left[ |\varphi'|^{1-\frac{p}{2}}\right]_{A_p(\partial \mathbb{D})}.$ The lower bound follows from a standard argument used for the Hilbert/Riesz transforms, see \cite{G}*{Theorem 7.4.7}. We include this argument for completeness. Let $\omega= |\varphi'|^{1-\frac{p}{2}}$, let $I_1=I_1(\zeta,r)$ be any boundary arc with  $r<\frac{2}{5}$, and let $I_2$ be as in Lemma \ref{ball separation}. First, note that by Lemma \ref{ball separation}, if $f>0$ is supported in $I_1$ we have the set containment

$$ I_2 \subseteq \left\{z \in \partial \mathbb{D}: |\mathcal{S}_0 f(z)|\geq C \fint_{I_1} f \, d \tilde{\theta} \right\}.$$

This then implies

\begin{align*}
\omega(I_2) & \leq \omega \left( \left\{z \in \partial \mathbb{D}: |\mathcal{S}_0 f(z)|\geq C \fint_{I_1} f \, d \tilde{\theta} \right\} \right) \\
& \leq \frac{\| \mathcal{S}_0\|^p_{L^p_\omega(\partial \mathbb{D}) \rightarrow L^p_\omega(\partial \mathbb{D})  }}{C^p \left(\fint_{I_1} f \, d \tilde{\theta}\right)^p} \|f\|^p_{L^p_\omega (\partial \mathbb{D})}\\
& = \frac{\| \mathcal{S}\|^p_{L^p(\partial \Omega) \rightarrow L^p(\partial \Omega)  }}{C^p \left(\fint_{I_1} f \, d \tilde{\theta}\right)^p} \|f\|^p_{L^p_\omega (\partial \mathbb{D})}.
\end{align*}

Assume without loss of generality that $\omega^{\frac{-1}{p-1}}$ is locally integrable (if not, replace $\omega$ by $\omega+\varepsilon$ and use a limiting argument). Then take $f= \omega^{\frac{-1}{p-1}} \chi_{I_1}$ and note that $\|f\|^p_{L^p_\omega (\partial \mathbb{D})}= \int_{I_1} \omega^{-\frac{1}{p-1}} \mathop{d \theta}$ and $|I_1|=|I_2|.$ Using these facts and rearranging the above display, we obtain

$$ \left(\fint_{I_2} \omega \, d \tilde{\theta}\right) \left(\fint_{I_1} \omega^{\frac{-1}{p-1}} \, d \tilde{\theta}\right)^{p-1} \leq \frac{\| \mathcal{S}\|^p_{L^p(\partial \Omega) \rightarrow L^p(\partial \Omega)  }}{C^p }.$$

Interchanging the roles of $I_1$ and $I_2$ , we obtain the analogous inequality:

$$ \left(\fint_{I_1} \omega \, d \tilde{\theta}\right) \left(\fint_{I_2} \omega^{\frac{-1}{p-1}} \, d \tilde{\theta}\right)^{p-1} \leq \frac{\| \mathcal{S}\|^p_{L^p(\partial \Omega) \rightarrow L^p(\partial \Omega)  }}{C^p }.$$

Finally, multiplying these two inequalities together and noting that $\left(\fint_{I_2} \omega \, d \tilde{\theta}\right) \left(\fint_{I_2} \omega^{\frac{-1}{p-1}} \, d\tilde{\theta}\right)^{p-1} \geq 1$ yields 

\begin{align*}
\left(\fint_{I_1} \omega \, d \tilde{\theta}\right) \left(\fint_{I_1} \omega^{\frac{-1}{p-1}}\, d\tilde{\theta}\right)^{p-1} & \leq  \left(\fint_{I_2} \omega \, d \tilde{\theta}\right) \left(\fint_{I_1} \omega^{\frac{-1}{p-1}} \, d\tilde{\theta}\right)^{p-1} \left(\fint_{I_1} \omega \, d \tilde{\theta}\right) \left(\fint_{I_2} \omega^{\frac{-1}{p-1}}d\tilde{\theta}\right)^{p-1} \\ 
& \leq \frac{\| \mathcal{S}\|^{2p}_{L^p(\partial \Omega) \rightarrow L^p(\partial \Omega)  }}{C^{2p} }.
\end{align*}

Rearranging this inequality and taking a supremum over intervals $I_1$ with small arc length then yields the desired result.

\end{proof}

\end{theorem}

\begin{rem}
Notice that the lower bound blows up at the endpoints $p=\frac{6}{5}$ and $p=6$, giving a proof that the Szeg\H{o} projection on this domain is only bounded for $\frac{6}{5}<p<6.$
\end{rem}

\begin{open}
As the reader may have noted, the dependences on $p$ in the upper and lower bounds appearing in Theorem \ref{SzegoSpecificQuant} are slightly different. It would be quite  interesting to determine $\|\mathcal{S}\|_{L^p(\Omega) \rightarrow L^p(\Omega)}$ exactly or asymptotically in this specific case, or else come up with a different example to show that the upper bound in Theorem \ref{SzegoMain} is in fact sharp. 

\end{open}

\subsection{Weak-Type Estimates}

We close the discussion of $L^p$ behavior with a couple of open questions. In particular, one could ask about the weak-type endpoint behavior of the operators $\mathcal{S}$ and $\mathcal{B}$, even in the unweighted case, when the projections are bounded in a limited range.  For instance, it would be quite interesting to see if endpoint weak-type estimates hold for the domains in Corollaries \ref{SzegoGraph}, \ref{SzegoLip}, \ref{BergmanGraph}, and \ref{BergmanLip} (for example, does $\mathcal{S}$ satisfy a weak-type $(\frac{4}{3},\frac{4}{3})$ estimate if $\Omega$ is a local graph domain?). The direct approach via conformal mapping seems to fail here, so the problem requires novel ideas. It is possible, though perhaps more difficult, that provided one solves the unweighted problem, one could use similar machinery to obtain weighted weak-type estimates. 

In contrast, for certain classes of nicer domains which admit full range $L^p$ bounds, including convex domains and domains with sufficient smoothness (for example, Dini-smooth domains), weak-type estimates have already been established (see \cite[Theorem 2.1]{Bek1}). However it would also be interesting to investigate the case $p=1$ when $\Omega$ is non-smooth, but nevertheless the Bergman and Szeg\H{o} are bounded in the full reflexive range, such as the vanishing chord-arc domains. Recall that $\Omega$ has the vanishing chord-arc property if the ratio $\frac{\sigma(pq)}{|p-q|}$ for $p,q \in \partial \Omega$ tends to $1$ as $p$ approaches $q$, where $\sigma(pq)$ denotes the length of the shortest arc joining $p$ and $q$.
 \begin{open}
 Investigate unweighted/weighted weak-type estimates for the Bergman/Szeg\H{o} projections at the endpoints of the $L^p$ boundedness ranges for local graph domains and Lipschitz domains. 
 
 \end{open}
 
  \begin{open}
 Investigate unweighted/weighted weak-type estimates for the Bergman/Szeg\H{o} projections at the endpoint $p=1$ for vanishing chord-arc domains. 
 
 \end{open}

\section{Bergman vs. Szeg\H{o}}\label{B vs S}
It is of interest to compare the Bergman and Szeg\H{o} projections via the Poisson extension operator $\mathcal{P}$ (see, for example \cite{LK}). Assuming that the boundary of $\Omega$ is sufficiently smooth, recall that the Poisson extension of a continuous function $f$ on $\partial \Omega$ is given by the (harmonic in $\Omega$) function

$$ \mathcal{P}f(z)=\int_{\partial \Omega} P(z,w) f(w) \, d \sigma(w), z \in \Omega,$$
where $P(z,w)$ denotes the Poisson kernel for $\Omega.$ Notice that the two operators $\mathcal{B} \mathcal{P}$ and $\mathcal{P} \mathcal{S}$ produce holomorphic functions in the interior of $\Omega$ from boundary data, and if the domain $\Omega$ is sufficiently regular, both will be bounded maps from $L^p( \partial \Omega)$ to $L^p(\Omega).$  In fact, in the case that $\Omega$ has sufficently smooth boundary, it can be shown that both $\mathcal{B} \mathcal{P}$ and $\mathcal{P} \mathcal{S}$ are bounded operators from $L^2(\partial \Omega)$ to the Sobolev space $L^2_{1/2}(\partial \Omega)$ \cite{LK} (there is also an $L^p$ version of the Sobolev result, but we will not concern ourselves with it here). 

Motivated by this Sobolev improvement, one might wonder about the compactness of the operators $\mathcal{B} \mathcal{P}$ and $\mathcal{P}\mathcal{S}$ as maps from $L^p(\partial \Omega)$ to $L^p(\Omega)$ (this may follow from some version of Sobolev embedding, but the author has not checked the details). However, this can be seen directly with rather minimal effort. We consider the case when $\Omega$ is Dini smooth, which means there is a parametrization $r(t)$ of the boundary Jordan curve $C$ so that $r'(t)$ is Dini continuous. Recall a function $f:\mathbb{R} \rightarrow \mathbb{C}$ is Dini continuous if $$ \int_{0}^{1} \frac{\omega(f,\delta)}{\delta} \, d \delta<\infty; \quad \omega(f,\delta):= \sup_{|x-y|\leq \delta} |f(x)-f(y)|.$$
Notice that Dini continuity is slightly weaker than Lipschitz/H\"{o}lder continuity and all Dini smooth domans are $C^1$, though the converse does not hold. We will use the key fact that if $\Omega$ is Dini-smooth and $\varphi: \mathbb{D} \rightarrow \Omega$ is a conformal map, then  $\varphi'$ extends continuously to $\bar{\mathbb{D}}$ and is non-vanishing on $\bar{\mathbb{D}}$ (see \cite[Theorem 3.5]{P}). We have the following result for this category of domains:

\begin{prop}\label{IndivCompact}
Let $\Omega \subset \mathbb{C}$ be a Dini-smooth bounded domain. Then the operators $\mathcal{B} \mathcal{P}$ and $\mathcal{P} \mathcal{S}$ are compact maps from $L^p( \partial \Omega)$ to $L^p(\Omega)$ for $1 < p<\infty$.
\begin{proof}
 We first note that the Bergman and Szeg\H{o} projections are bounded on $L^p$ for $1<p<\infty$ in this situation by Theorems \ref{SzegoQuant} and \ref{BergmanQuant}, for instance. Therefore, it suffices to show that the Poisson extension operator $\mathcal{P}$ is compact from $L^p( \partial \Omega)$ to $L^p(\Omega)$.
 
 Indeed, this can be seen directly using the formula $\mathcal{P}=C_{\varphi^{-1}} \mathcal{P}_0 C_{\varphi}$ given in \cite[pp. 972]{LK}, where $C_{\varphi}$ denotes the composition operator with symbol $\varphi$ (which denotes the conformal map as before). This formula apriori holds for continuous functions on $\partial \Omega$, but it is easy to show it meaningfully extends to $L^p(\partial \Omega).$ To begin with, since $|\varphi'|$ is bounded above and below on $\bar{\mathbb{D}}$, the maps $C_{\varphi}$ and $C_{\varphi}^{-1}$ are bounded from $L^p(\partial \Omega)$ to $L^p(\partial \mathbb{D})$  and $L^p(\mathbb{D})$ to $L^p(\Omega)$, respectively. We have therefore reduced the problem to showing that $\mathcal{P}_0$ is compact from $L^p(\partial \mathbb{D})$ to $L^p(\mathbb{D})$ for $1 <p<\infty$ (actually, the argument will show that one can include the endpoint $p=1$ for this statement). 
 
 To this end, take $f_n$ converging weakly to zero in $L^p(\partial \mathbb{D}).$ To prove the operator $\mathcal{P}$ is compact, it suffices to show that for all $\varepsilon>0$, there exists an $N \in \mathbb{N}$ so that for all $n>N$, we have
\begin{equation} \|\mathcal{P} f_n \|_{L^p(\mathbb{D})}^p< \varepsilon.      \label{smallLp} \end{equation}
 
 Since the $f_n$ converge weakly, we know that there exists a $B>0$ such $\sup_{n} \|f_n\|_{L^p(\partial \mathbb{D})}\leq B.$ We first claim that $\mathcal{P}f_n$ converges to $0$ pointwise on $\mathbb{D}.$ To see this, write
 \begin{equation}
 |\mathcal{P} f_n(re^{\ii \theta})| = \left| \frac{1}{2 \pi} \int_{0}^{2 \pi} \frac{(1-r^2) f_n(e^{\ii t})}{|e^{\ii t}-r e^{\ii \theta}|^2}  \, dt \right|. \label{PoissonForm}
\end{equation} 
For fixed $r, \theta$, the function $g_{r, \theta}(e^{\ii t})=\frac{(1-r^2) }{|e^{\ii t}-r e^{\ii \theta}|^2} $ clearly belongs to $L^{p'}(\partial \mathbb{D})$ (in fact, it belongs to $L^{\infty}(\partial \mathbb{D})$),  so then by weak convergence it is clear that \eqref{PoissonForm} converges to $0$ as $n \rightarrow \infty.$ We have thus established pointwise convergence.  

We now claim that for all $\varepsilon>0$ we can choose $r_0$ sufficiently close to $1$ so that

\begin{equation} 
\sup_{n} \frac{1}{\pi} \int_{0}^{2\pi} \int_{r_0}^1 |\mathcal{P}f_n(re^{\ii \theta})|^p r \, dr \, d\theta< \frac{\varepsilon}{2}. \label{ClaimMaximal}
\end{equation} 

Assuming claim \eqref{ClaimMaximal}, we may write, for all $n \in \mathbb{N}$:

\begin{align*}
\|\mathcal{P} f_n \|_{L^p(\mathbb{D})}^p & = \frac{1}{\pi} \int_{0}^{2\pi} \int_{0}^1 |\mathcal{P}f_n(re^{\ii \theta})|^p r \, dr \, d\theta \\
& =  \frac{1}{\pi} \int_{0}^{2\pi} \int_{0}^{r_0} |\mathcal{P}f_n(re^{\ii \theta})|^p r \, dr \, d\theta + \frac{1}{ \pi} \int_{0}^{2\pi} \int_{r_0}^{1} |\mathcal{P}f_n(re^{\ii \theta})|^p r \, dr \, d\theta \\
& < \frac{1}{ \pi} \int_{0}^{2\pi} \int_{0}^{r_0} |\mathcal{P}f_n(re^{\ii \theta})|^p r \, dr \, d\theta + \frac{\varepsilon}{2}.
\end{align*}
Notice that if $0<r<r_0$, we have the uniform estimate
$$ |\mathcal{P}f_n(re^{\ii \theta})|^p r \leq \frac{r_0}{(1-r_0)^{2p}} \left|\frac{1}{2 \pi} \int_{0}^{2 \pi} f_n(e^{\ii t}) \, dt \right|^p \leq \frac{r_0 B^p}{(1-r_0)^{2p}}.  $$
Now, using for example the Dominated Convergence Theorem with constant dominating function $F(r,\theta)=\frac{r_0 B^p}{(1-r_0)^{2p}} $ we may choose $N$ sufficiently large so that for all $n>N$, we have
$$ \frac{1}{ \pi} \int_{0}^{2\pi} \int_{0}^{r_0} |\mathcal{P}f_n(re^{\ii \theta})|^p r \, dr \, d\theta< \frac{\varepsilon}{2}.$$
Assuming claim \eqref{ClaimMaximal}, we have established \eqref{smallLp}.

It remains to prove \eqref{ClaimMaximal}. This is a straightforward consequence of Minkowski's inequality, which yields

$$\sup_{0<r<1} \frac{1}{2 \pi} \int_{0}^{2\pi} |\mathcal{P}f_n(re^{\ii \theta})|^p   \, d\theta \leq \|f_n\|_{L^p(\partial \mathbb{D})}^p \leq B^p.$$
(see, for example, \cite[pp. 14]{Gar}). 

\end{proof}
\end{prop}

\begin{remark}
The author does not know of a reference for this fact concerning the compactness of the Poisson extension on the disk, but it is probably well known to the experts. The result appears to be mentioned in passing in \cite{Englis} in the context of a bounded domain $\Omega \subset \mathbb{R}^n$ with $C^\infty$ smooth boundary and for $p=2$. 

\end{remark}

\begin{remark}
The reason for the Dini-smooth hypothesis for this result is that this guarantees the boundedness of the Bergman and Szeg\H{o} projections for $1<p<\infty$, and also the composition maps $C_{\varphi}$ and $C_{\varphi^{-1}}$, as $\varphi'$ is bounded above and below on $\bar{\mathbb{D}}.$ However, it is very plausible that this result extends to domains of class $C^1$. In particular, it is known that $\mathcal{B}$ and $\mathcal{S}$ are bounded on $L^p$ for $1<p<\infty$ for $C^1$ domains (see \cite{LS1}), and in fact as we mentioned this can be weakened to vanishing chord-arc domains. One would have to check the compactness of $\mathcal{P}$ in this situation, which is not immediate as $C_\varphi$ and $C_{\varphi^{-1}}$ may not be bounded maps. 

\end{remark}

A more interesting question, perhaps, is to study how the boundary regularity improves the mapping properties of the difference operator $\mathcal{B} \mathcal{P}- \mathcal{P} \mathcal{S}.$ For domains with some degree of smoothness, one would expect the difference operator $\mathcal{B} \mathcal{P}-\mathcal{P}\mathcal{S}$ to be better behaved than each of the individual operators in an appropriate sense due to cancellation. Indeed, in the special case of the unit disk $\mathbb{D}$, one actually shows that $\mathcal{B}_0 \mathcal{P}_0=\mathcal{P}_0 \mathcal{S}_0$, but clearly this complete cancellation cannot be expected to hold for general domains. 

Given a positive integer $k$ and $0<\alpha < 1$, we say a domain $\Omega$ has $C^{k, \alpha}$ boundary if its boundary can be locally represented as the graph of a $C^{k, \alpha}$ function. It is known that if $\varphi$ is a conformal map associated to a domain $\Omega$ with $C^{k,\alpha}$ boundary, then $\varphi \in C^{k,\alpha}(\bar{\mathbb{D}})$ and $\varphi$ is nonvanishing on $\bar{\mathbb{D}}.$ The authors in \cite{LK} analyzed the mapping properties of the difference operator for $C^{k,\alpha}$ domains with $k \geq 2$ in terms of a gain in the Sobolev or Lipschitz space scale. In particular, they proved that the difference operator $\mathcal{B}\mathcal{P}-\mathcal{P} \mathcal{S}$ maps $L^2(\partial \Omega)$ to the Sobolev space $L^2_{3/2}(\Omega)$, representing a gain of a full derivative in the Sobolev scale compared to the operators individually. The authors also showed that $\mathcal{B} \mathcal{P}-\mathcal{P}\mathcal{S}$ is a bounded map from $L^\infty(\partial \Omega)$ to $\Lambda_1(\Omega)$, where $\Lambda_1(\Omega)$ is a H\"{o}lder-Zygmund space \cite[Theorem 1.2]{LK}. We will not concern ourselves with the precise definition of $\Lambda_1(\Omega)$ here because this result is not our focus.  
 
 Motivated by Proposition \ref{IndivCompact}, we might wonder about the compactness of the individual operators $\mathcal{B} \mathcal{P}$ and $\mathcal{P} \mathcal{S}$ at the endpoints $p=1$ and $p =\infty.$ However, in contrast to the other values of $p$, it is clear that the operator $\mathcal{P}_0$ is not compact from $L^\infty(\partial \mathbb{D})$ to $L^\infty(\mathbb{D})$ (it is in fact an isometry), and the projections $\mathcal{B}$ and $\mathcal{S}$ are in general not bounded at either of the endpoints. Thus, the compactness (even the boundedness) of $\mathcal{B} \mathcal{P}$ and $\mathcal{P}\mathcal{S}$ cannot be expected to hold at the endpoints.
 
In a similar spirit as \cite{LK}, we can apply a characterization of compact integral operators on $L^1$ and $L^\infty$ given in \cite{E} in a simple way to show that the difference operator $\mathcal{B}\mathcal{P}-\mathcal{P} \mathcal{S}$ has the improved property of compactness on $L^1$ and $L^\infty$ rather than just $L^p$, if we assume the boundary is just slightly more regular. We remark that our assumptions on regularity are less than those in \cite{LK}; this is heuristically because we are proving compactness rather than quantitative Sobolev results.

\begin{theorem}\label{CompactDiffInf} Suppose $\Omega$ has $C^{1,\alpha}$ boundary and $0<\alpha < 1.$ Then $\mathcal{B} \mathcal{P}- \mathcal{P} \mathcal{S}$ is a compact
operator from $L^1(\partial \Omega)$ to $L^1(\Omega)$ and $L^\infty(\partial \Omega)$ to $L^\infty(\Omega).$

\begin{proof}  
First, using the formula given in \cite[pp. 974]{LK}, we may write

$$\mathcal{B}\mathcal{P}-\mathcal{P}\mathcal{S}= \tau^{-1} \circ \left( [\mathcal{B}_0,\mathcal{M}_{\varphi'}]\mathcal{P}_0 \mathcal{M}_{(\varphi ')^{-1/2}}+ \mathcal{M}_{\varphi'}\mathcal{P}_0[\mathcal{S}_0,\mathcal{M}_{(\varphi')^{-1/2}}]\right) \circ \tau_{1/2},$$
where the notation $[A,B]$ means the commutator $AB-BA$ and $\mathcal{M}_f$ denotes the multiplication operator with symbol $f.$

Notice that, on account of the fact that $\Omega$ is $C^{1,\alpha}$, the function $\varphi'$ is bounded above and below on $\bar{\mathbb{D}}.$ It is then clear from Lemmas \ref{transfer to disc Szego} and \ref{TransferDiscBergman} that $\tau_{1/2}$ is a linear homeomorphism from $L^p(\partial \Omega)$ to $L^p(\partial \mathbb{D})$, and similarly $\tau^{-1}$ is a linear homeomorphism from $L^p(\mathbb{D})$ to $L^p(\Omega).$ Moreover, the operators $\mathcal{M}_{\varphi'}$ and $\mathcal{M}_{(\varphi')^{-1/2}}$ are evidently bounded on $L^p(\mathbb{D})$ and $L^p(\partial \mathbb{D})$ respectively. We also know the Poisson extension $\mathcal{P}_0$ is a bounded operator from $L^p(\partial \mathbb{D})$ to $L^p(\mathbb{D})$ for $1\leq p \leq \infty$.

Therefore, it clearly suffices to prove the compactness of the two commutators $[\mathcal{B}_0,\mathcal{M}_{\varphi'}]$ and $[\mathcal{S}_0,\mathcal{M}_{(\varphi')^{-1/2}}].$ We only prove the compactness of $[\mathcal{B}_0,\mathcal{M}_{\varphi'}]$, as the other commutator is similar. Moreover, we only prove compactness for $p=\infty$; the result for $p=1$ will follow using the same argument.

Let $k_z(w)=K(z,w)=\frac{(\varphi'(z)-\varphi'(w))}{(1-\overline{w}z)^2}$ denote the kernel of the operator $[\mathcal{B}_0,\mathcal{M}_{\varphi'}]$. It is easy to check that $K(z,w)$ is uniformly continuous on the set $\{(z,w) \in \bar{\mathbb{D}} \times \bar{\mathbb{D}}: |z-w| \geq r\}$ for each $r>0.$  We have, using our hypothesis that $\Omega$ is $C^{1,\alpha}$:

\begin{eqnarray*}
|k_z(w)| & = & \frac{|\varphi'(z)-\varphi'(w)|}{|1-\overline{w}z|^2}\\
& \leq &  C_{\varphi} \frac{|z-w|^\alpha}{|1-\overline{w}z|^2}\\
& \leq & \frac{C_\varphi}{|z-w|^{2- \alpha}}.
\end{eqnarray*}
It is thus obvious
$$\sup_{z \in \mathbb{D}} \int_{\mathbb{D}}|k_z(w)| \mathop{dA(w)}<\infty.$$

Moreover, an easy computation shows

\begin{equation} \int_{D_\delta(z)} \frac{1}{|z-w|^{2-\alpha}} \mathop{d A(w)} = \frac{2}{\alpha} \delta^\alpha. \label{integratekernel} \end{equation}

By a characterization given in \cite[Corollary 5.1]{E}, it is sufficient to show that the set of functions $\{k_z\}_{z \in \mathbb{D}}$ is relatively compact in $L^1(\mathbb{D})$. This can be checked using a version of the Riesz-Kolmogorov theorem given in \cite[Lemma 1]{IK}: 
$$\lim_{r \rightarrow 0} \sup_{z \in \mathbb{D}} \int_{\mathbb{D}}\left|k_z(w)-\fint_{D_r(w)}k_z \, dA \right| \mathop{dA(w)}=0.$$

Fix an arbitrary $\varepsilon>0$. Let $\delta^\alpha=\varepsilon$ and let $r<\delta$. We will show that we can make a sufficiently small choice of $r$ so

$$\sup_{z \in \mathbb{D}} \int_{\mathbb{D}} \left|k_z(w)-\fint_{D_r(w)}k_z \, dA\right| \mathop{dA(w)} \lesssim \delta^\alpha.$$

First, we show 

$$\sup_{z \in \mathbb{D}} \int_{D_\delta(z)} \left|k_z(w)-\fint_{D_r(w)}k_z \, dA\right| \mathop{d A(w)} \lesssim \delta^\alpha.$$
To see this, note that by \eqref{integratekernel} and the triangle inequality it is clearly enough to show 

$$ \sup_{z \in \mathbb{D}} \int_{D_\delta(z)} \left|\fint_{D_r(w)}k_z \, dA\right| \mathop{d A(w)} \lesssim \delta^\alpha                           .$$
For fixed $z \in \mathbb{D}$, let $G_z=\{w:|z-w| \geq 2r\}.$ Note that if $w \in G_z$ and $\zeta \in D_r(w),$ then $|z-w| \leq 2 |z-\zeta|.$ We further split the region of integration and estimate for fixed $z$:

\begin{align*}
 \int_{D_\delta(z) \cap G_z} \left|\fint_{D_r(w)}k_z \, dA \right| \mathop{dA(w)} & \leq \int_{D_\delta(z) \cap G_z} \left(\frac{1}{|D_r(w)|}\int_{D_r(w)} |k_z(\zeta)| \, dA(\zeta)\right) \, dA(w) \\
 & \leq 2^{2-\alpha} C_\varphi \int_{D_\delta(z) \cap G_z} \left(\frac{1}{|D_r(w)|}\int_{D_r(w)} \frac{1}{|z-w|^{2-\alpha}} \, dA(\zeta)\right) \, dA(w)\\
 & \leq 2^{2-\alpha} C_\varphi \int_{D_\delta(z)} \frac{1}{|z-w|^{2-\alpha}} \, dA(w)\\
 & \leq \frac{8}{\alpha} C_\varphi \delta^\alpha .
\end{align*}
On the other hand, 

\begin{align*}
 \int_{D_\delta(z) \cap G_z^c} \left|\fint_{D_r(w)}k_z \, dA\right| \mathop{dA(w)} & \leq \int_{D_{2r}(z)} \left(\frac{1}{|D_r(w)|}\int_{D_{3r}(z)} |k_z(\zeta)| \, dA(\zeta)\right) \, dA(w)\\
 & \leq \frac{24}{\alpha} C_\varphi r^\alpha\\
 & \lesssim \delta^\alpha,
\end{align*}
establishing the first claim.

For the other region of integration, by uniform continuity we can choose $r$ small enough so 

$$\sup_{z \in \mathbb{D}} \int_{\mathbb{D} \setminus B(z, \delta)} \left|k_z(w)-\fint_{D_r(w)}k_z \, dA\right| \mathop{dA(w)}< \delta^\alpha,$$
establishing the result.

\end{proof}
\end{theorem}

The proof above makes use of fairly strong hypotheses on the boundary smoothness of $\Omega.$ It is natural to ask whether these hypothesis can be weakened, especially since $\mathcal{B}$ and $\mathcal{S}$ are both bounded operators on $L^p$ for all $p$ on domains that are less regular. 

\begin{open}
Determine if there is a class larger than the $C^{k, \alpha}$ domains $(k \geq 1)$) so that $\mathcal{B} \mathcal{P}-\mathcal{P}\mathcal{S}$ is compact from $L^p(\partial \Omega)$ to $L^p(\Omega)$ for $1 \leq p\leq \infty.$ In particular, does the result hold $\Omega$ is Dini-smooth, or even better when $\Omega$ is class $C^1$ (or less smooth)?

\end{open}

\bibliographystyle{amsplain}

\end{document}